\documentclass[12pt]{article}

\usepackage{amssymb}
\usepackage{latexsym}
\usepackage{amsmath}
\usepackage{indentfirst}
\usepackage{color}
\usepackage{thmtools}
\usepackage{thm-restate}
\usepackage{hyperref}
\usepackage{ulem}

\title{A note on degree conditions for Ramsey goodness of trees}
\author{Zhidan Luo{\thanks{School of Mathematics and Statistics, Hainan University, Haikou 570228, P. R. China. Research supported in part by National Natural Science Foundation of China (No.12401449), the Hainan Provincial Natural Science Foundation of China (No.125QN209) and the Research Foundation Project of Hainan University (No. KYQD(ZR)-23155). Email: luodan@hainanu.edu.cn}}
\and Yuejian Peng \thanks{Corresponding author. School of Mathematics, Hunan University, Changsha 410082, P. R. China. Research supported in part by National Natural Science Foundation of China (No.12571363) and National Natural Science Foundation of Hunan Province (Grant No.2025JJ30003) Email: ypeng1@hnu.edu.cn}}
\date{}

\newtheorem{theo}{Theorem}[section]

\newtheorem{remark}[theo]{Remark}
\newtheorem{lemma}[theo]{Lemma}

\newtheorem{conj}[theo]{Conjecture}

\def\q{\hspace*{\fill}$\Box$\medskip}

\begin{document}
\maketitle

\begin{abstract}
  For given graphs $G_{1}, G_{2}$ and $G$, let $G\rightarrow (G_{1}, G_{2})$ denote that each red-blue-coloring of $E(G)$ yields a red copy of $G_{1}$ or a blue copy of $G_{2}$. Arag{\~a}o,  Marciano and  Mendon{\c c}a   [L. Arag{\~a}o, J. Pedro Marciano and W. Mendon{\c c}a, Degree conditions for Ramsey goodness of paths, {\it European Journal of Combinatorics}, {\bf 124} (2025), 104082] proved the following. Let $G$ be a graph on $N\geq (n- 1)(m- 1)+ 1$ vertices. If $\delta(G)\geq N- \lceil n/2\rceil$, then $G\rightarrow (P_{n}, K_{m})$, where $P_{n}$ is a tree on $n$ vertices.  In this note, we generalize $P_{n}$ to any tree $T_{n}$ with $n$ vertices, and improve the lower bound of $\delta(G)$. We further improve the lower bound when $T_{n}\neq K_{1, n- 1}$, which partially confirms their conjecture.


\end{abstract}

\section{Introduction}
  For graphs $G_{1}, \dots, G_{t}$ and $G$, let $G\rightarrow (G_{1}, \dots, G_{t})$ denote that any $t$-coloring of $E(G)$ yields a monochromatic copy of $G_{i}$ in color $i$ for some $i\in [t]$. The Ramsey number $r(G_{1}, \dots, G_{t})$ is the minimum positive integer $N$ such that $K_{N}\rightarrow (G_{1}, \dots, G_{t})$. Let $v(G)$ and $\chi(G)$ be the number of vertices and the chromatic number of $G$, respectively. Let $s(G)$ be the surplus of $G$, which is the size of a minimum color class taken over all proper colorings of $G$ with $\chi(G)$ colors. Let $T_{n}, K_{1, n- 1}$ and $P_{n}$ denote a tree, a star and a path on $n$ vertices, respectively.

  In 1981, Burr \cite{B} observed the following. For a connected graph $G$ and a graph $H$ with $v(G)\geq s(H)$, $r(G, H) \geq (v(G)- 1)(\chi(H)- 1)+ s(H)$ (let $G'$ be a red copy of the complete $\chi(H)$-partite graph such that all parts have sizes $v(G)- 1$ except one with size $s(H)- 1$. Moreover, color all edges in each part by blue. The resulting graph contains neither a red copy of $G$ nor a blue copy of $H$). We say that $G$ is $H$-good if the equality holds. There is a classical result from Chv{\'a}tal \cite{C} showing that $T_{n}$ is $K_{m}$-good for all positive integers $n$ and $m$, that is, $r(T_{n}, K_{m})= (n- 1)(m- 1)+ 1$. Recently, a similar question has attracted a lot of attention, which originated from Schelp \cite{S}. Let $G$ be a graph on $N\geq r(G_{1}, G_{2})$ vertices, what minimum degree condition guarantees that $G\rightarrow (G_{1}, G_{2})$? Recently, Arag{\~a}o, Pedro Marciano and Mendon{\c c}a studied $P_{n}$ versus $K_{m}$.

  \begin{theo}[Arag{\~a}o, Pedro Marciano and Mendon{\c c}a \cite{AMM}]\label{theo1.1}
    Let $G$ be a graph on $N\geq r(P_{n}, K_{m})= (n- 1)(m- 1)+ 1$ vertices. If $\delta(G)\geq N- \lceil n/2\rceil$, then $G\rightarrow (P_{n}, K_{m})$.
  \end{theo}

  \noindent Moreover, they conjectured the following and showed that their conjecture holds for $m= 3$.

  \begin{conj}[Arag{\~a}o, Pedro Marciano and Mendon{\c c}a \cite{AMM}]\label{conj1.2}
    For given positive integers $n, m, N$ and a non-negative integer $t$, assume that $(t+ 1)(n- 1)(m- 1)+ 1\leq N\leq (t+ 2)(n- 1)(m- 1)$. Let $G$ be a graph on $N$ vertices with $\delta(G)\geq N- \left\lceil \frac{t+ 1}{t+ 2}\left\lceil \frac{N}{m- 1}\right\rceil\right\rceil$. Then $G\rightarrow (P_{n}, K_{m})$.
  \end{conj}

   If Conjecture \ref{conj1.2} is true, then the lower bound of $\delta(G)$ in Conjecture \ref{conj1.2} is tight as evidenced by the following construction.

   {\bf Construction I.} Assume that $(t+ 1)(n- 1)(m- 1)+ 1\leq N\leq (t+ 2)(n- 1)(m- 1)$. Let $G'$ be a complete almost balanced $(m- 1)$-partite graph (the size of each part differs by at most one) on $N$ vertices whose edges are colored by blue. Moreover, in each part, we embed a red copy of the graph consisting of the union of $t+ 2$ cliques such that the sizes of each clique differ by at most one. Denote the resulting graph by $G$. Note that $G\not\rightarrow (P_{n}, K_{m})$ (in fact, $G\not\rightarrow (T_{n}, K_{m})$) and $\delta(G)= N- \lceil\frac{N}{m- 1}\rceil+ \lfloor \frac{1}{t+ 2}\lceil \frac{N}{m- 1}\rceil\rfloor- 1= N- \lceil \frac{t+ 1}{t+ 2}\lceil \frac{N}{m- 1}\rceil\rceil- 1$.

  In this note, we generalize $P_{n}$ to $T_{n}$, improve the lower bound of $\delta(G)$ in Theorem \ref{theo1.1} for all $m$ and $t\geq 1$, and partially confirm Conjecture \ref{conj1.2}. We first show the following result.

  \begin{theo}\label{theo1.3}
    For given positive integers $n, m, N$ and a non-negative integer $t$, assume that $(t+ 1)(n- 1)(m- 1)+ 1\leq N\leq (t+ 2)(n- 1)(m- 1)$. Let $G$ be a graph on $N$ vertices with $\delta(G)\geq N- t(n- 1)- 1$. Then $G\rightarrow (T_{n}, K_{m})$.
  \end{theo}

  Let $\Delta= \Delta(T_{n})$. We remark that the lower bound on $\delta(G)$ in Conjecture \ref{conj1.2} is not tight for $T_{n}$ when $\Delta> \lfloor \frac{1}{t+ 2}\lceil \frac{N}{m- 1}\rceil\rfloor$ and $n$ is even as evidenced by the following construction.

   {\bf Construction II.}  Let $G'$ be a complete almost balanced $(m- 1)$-partite graph on $N$ vertices whose edges are colored by blue. In each part, we embed a red copy of  a $(\Delta- 1)$-regular graph or a red copy of an almost $(\Delta- 1)$-regular graph (all vertices with degree $\Delta- 1$ except one vertex with degree $\Delta- 2$). Denote the resulting graph by $H$. Note that $H\not\rightarrow (T_{n}, K_{m})$ and $\delta(H)= N- \lceil\frac{N}{m- 1}\rceil+ \Delta- \epsilon$, where $\epsilon= 1$ if we can replace it with a $(\Delta- 1)$-regular graph, and otherwise $\epsilon= 2$.

  \begin{conj}\label{conj1.4}
    For given positive integers $n, m, N$ and a non-negative integer $t$, assume that $(t+ 1)(n- 1)(m- 1)+ 1\leq N\leq (t+ 2)(n- 1)(m- 1)$. Let $G$ be a graph on $N$ vertices with $\delta(G)\geq N- \lceil\frac{N}{m- 1}\rceil+ \max\{\Delta- \epsilon+ 1, \lfloor \frac{1}{t+ 2}\lceil \frac{N}{m- 1}\rceil\rfloor\}$. Then $G\rightarrow (T_{n}, K_{m})$.
  \end{conj}

   By {\bf Construction I} and {\bf Construction II},  if Conjecture \ref{conj1.4} is true, then the lower bound of $\delta(G)$ on Conjecture \ref{conj1.4} is tight.
  Applying Theorem \ref{theo1.3},  Conjecture \ref{conj1.4} holds for $n$ even, $T_{n}= K_{1, n- 1}$ and $(t+ 1)(n- 1)(m- 1)+ 1\leq N\leq (t+ 1)(n- 1)(m- 1)+ m-1$.

  In Section 2, we will prove Theorem \ref{theo1.3} by determining  the Ramsey number $r(K_{1, \ell}, T_{n}, K_{m})$ for $\ell= t(n- 1)+ 1$. In Section 3, we strengthen Theorem \ref{theo1.3} if $T_{n}\neq K_{1, n- 1}$. A result in Section 3 implies that Conjecture \ref{conj1.2} holds for $3\leq n\leq t+ 4, m- 2\not\equiv 0\pmod{n- 1}$ and $(t+ 1)(n- 1)(m- 1)+ 1\leq N\leq (t+ 1)(n- 1)(m- 1)+ m- 1$.

\section{Proof of Theorem \ref{theo1.3}}
  We will  prove Theorem \ref{theo1.3} by establishing the following result on the Ramsey number.

  \begin{theo}\label{theo2.1}
    If $\ell= t(n- 1)+ 1$, then
    $$r(K_{1, \ell}, T_{n}, K_{m})= (\ell+ n- 2)(m- 1)+ 1= (t+ 1)(n- 1)(m- 1)+ 1.$$
  \end{theo}

  \noindent We will prove Theorem \ref{theo2.1} later. Let us first apply Theorem \ref{theo2.1} to prove Theorem \ref{theo1.3}.\\

  \noindent{\it Proof of Theorem \ref{theo1.3}.\quad} Let $(t+ 1)(n- 1)(m- 1)+ 1\leq N\leq (t+ 2)(n- 1)(m- 1)$, and  let $G$ be a graph on $N$ vertices with $\delta(G)\geq N- t(n- 1)- 1$. We  need to show that $G\rightarrow (T_{n}, K_{m})$. Color $E(G)$ arbitrarily by blue and green and color $E(\overline{G})$ by red. Then we get a red-blue-green edge-colored $K_{N}$. By Theorem 2.1, $N\geq r(K_{1, \ell}, T_{n}, K_{m})$, where $\ell= t(n- 1)+1$. Consequently, $K_{N}\rightarrow (K_{1, \ell}, T_{n}, K_{m})$. Note that $\Delta(\overline{G})\leq N-1- \delta(G)\leq t(n- 1)< \ell$. Thus, there is no red copy of $K_{1, \ell}$. Consequently, $G\rightarrow (T_{n}, K_{m})$.\q

  In what follows, we divide the proof of Theorem \ref{theo2.1} into two parts.

  \begin{theo}\label{theo2.2}
    Let $r'= r(K_{1, \ell}, T_{n})$. Then $$r(K_{1, \ell}, T_{n}, K_{m})\geq (r'- 1)(m- 1)+ 1.$$
  \end{theo}
  \begin{proof}
    Let $G'$ be a red-blue edge-colored $K_{r'- 1}$ such that $G'\not\rightarrow (K_{1, \ell}, T_{n})$ and let $G''$ be a green copy of $K_{m- 1}$. Replace each vertex of $G''$ with a copy of $G'$. Each edge of $G''$ expands into a green copy of $K_{r'- 1, r'- 1}$. Denote the resulting graph by $G$. Note that $G$ is a red-blue-green edge-colored $K_{(r'- 1)(m- 1)}$ such that $G\not\rightarrow (K_{1, \ell}, T_{n}, K_{m})$. Consequently, $r(K_{1, \ell}, T_{n}, K_{m})\geq (r'- 1)(m- 1)+ 1$.\q
  \end{proof}

  \noindent We still need a result from Burr to complete the proof of Theorem \ref{theo2.1}.

  \begin{theo}[Burr \cite{B1}]\label{theo2.3}
    If $\ell= t(n- 1)+ 1$, then $r(K_{1, \ell}, T_{n})= \ell+ n- 1$.
  \end{theo}

  \noindent{\it Proof of Theorem \ref{theo2.1}:\quad} Let $N= (\ell+ n- 2)(m- 1)+ 1$. We only need to show that $r(K_{1, \ell}, T_{n}, K_{m})\leq N$ by Theorem \ref{theo2.2} and Theorem \ref{theo2.3}. Color $E(K_{N})$ arbitrarily by red, blue and green, and denote the resulting graph by $H$. Moreover, let $H_{r}, H_{b}$ and $H_{g}$ be the graph induced by all edges in red, blue and green, respectively. We will show that $H\rightarrow (K_{1, \ell}, T_{n}, K_{m})$ using induction on $m$. It is true for $m= 2$ since $r(K_{1, \ell}, T_{n}, K_{2})= r(K_{1, \ell}, T_{n})= \ell+ n- 1$ by Theorem \ref{theo2.3}. If $\Delta(H_{g})\geq (\ell+ n- 2)(m- 2)+ 1$, then we are done since $r(K_{1, \ell}, T_{n}, K_{m- 1})= (\ell+ n- 2)(m- 2)+ 1$ by the induction hypothesis. Consequently, we may assume that $\Delta(H_{g})\leq (\ell+ n- 2)(m- 2)$, and thus, $\delta(H_{r}\cup H_{b})\geq \ell+ n- 2$. If $\Delta(H_{r})\geq \ell$, then we are done. Consequently, we may assume that $\Delta(H_{r})\leq \ell- 1$, and thus, $\delta(H_{b})\geq n- 1$. Consequently, we can embed $T_{n}$ into $H_{b}$ greedily since $\Delta(T_{n})\leq n- 1$ and $T_{n}$ is $1$-degenerate (we can label $V(T_{n})$ as $\{v_{1}, \dots, v_{n}\}$ such that $|N(v_{i})\cap \{v_{1}, \dots, v_{i- 1}\}|\leq 1$ for each $i\in [n]\backslash [1]$).\q

\section{Beyond Theorem \ref{theo1.3}}
  In fact, we can improve the lower bound on $\delta(G)$ in Theorem \ref{theo1.3} when $T_{n}$ is not a star. In 1995, Guo and Volkmann proved a result similar to Theorem \ref{theo2.3} as follows.

  \begin{theo}[Guo and Volkmann \cite{GV}]\label{theo3.1}
    If $\ell= t(n- 1)+ 2$ and $T_{n}\neq K_{1, n- 1}$, then $r(K_{1, \ell}, T_{n})= \ell+ n- 2$.
  \end{theo}

  \noindent The core of the proof of Theorem \ref{theo3.1} is the following.

  \begin{lemma}[Guo and Volkmann \cite{GV}]\label{lemma3.2}
    Let $G$ be a connected graph on at least $n$ vertices with $\delta(G)\geq n- 2$. If $T_{n}\neq K_{1, n- 1}$, then we can embed $T_{n}$ into $G$.
  \end{lemma}

  \noindent Similar to the proof of Theorem \ref{theo2.1}, applying Lemma \ref{lemma3.2}, we can prove the following.

  \begin{theo}\label{theo3.3}
    If $\ell= t(n- 1)+ 2, m- 2\not\equiv 0\pmod{n- 1}$ and $T_{n}\neq K_{1, n- 1}$, then
    $$r(K_{1, \ell}, T_{n}, K_{m})= (\ell+ n- 3)(m- 1)+ 1= (t+ 1)(n- 1)(m- 1)+ 1.$$
  \end{theo}

  \noindent Similar to the proof of Theorem \ref{theo1.3}, the following is true by applying Theorem \ref{theo3.3}.

  \begin{theo}\label{theo3.4}
    For given positive integers $n, m, N$ and a non-negative integer $t$, assume that $(t+ 1)(n- 1)(m- 1)+ 1\leq N\leq (t+ 2)(n- 1)(m- 1)$. Let $G$ be a graph of $N$ vertices with $\delta(G)\geq N- t(n- 1)- 2$. If $m- 2\not\equiv 0\pmod{n- 1}$ and $T_{n}\neq K_{1, n- 1}$, then $G\rightarrow (T_{n}, K_{m})$.
  \end{theo}

  \begin{remark}
    Applying Theorem \ref{theo3.4} and direct calculations, we can obtain that Conjecture \ref{conj1.2} holds for $3\leq n\leq t+ 4, m- 2\not\equiv 0\pmod{n- 1}$ and $(t+ 1)(n- 1)(m- 1)+ 1\leq N\leq (t+ 1)(n- 1)(m- 1)+ m- 1$.
  \end{remark}

  In fact, Yan and Peng \cite{YP} strengthened Lemma \ref{lemma3.2} (Theorem 1.2 in \cite{YP}). They showed that if $G$ is a connected graph on at least $n$ vertices with $\delta(G)\geq n- 3$, then all except some special trees with $n$ vertices can be embedded into $G$ (see Theorem 1.2 in \cite{YP}). Applying this result and the same discussion as before, we can confirm Conjecture \ref{conj1.2} for more values of $n, m$ and $N$. Moreover, Conjecture \ref{conj1.2} may be verified by determining all the values of $r(K_{1, \ell}, P_{n}, K_{m})$.

\end{document}